\documentclass[oneside, 10pt]{amsart}

\usepackage{amsmath, amsfonts, amssymb, amsthm, graphics, times}

\newtheorem{theorem}{Theorem}[section]
\newtheorem{lemma}[theorem]{Lemma}
\newtheorem{corollary}[theorem]{Corollary}
\newtheorem{proposition}[theorem]{Proposition}
\newtheorem{conjecture}[theorem]{Conjecture}

\theoremstyle{definition}

\newtheorem{remark}[theorem]{Remark}

\numberwithin{equation}{section}

\begin{document}

\title[Solution to the Khavinson problem]{Solution to the Khavinson problem near the boundary of the unit ball}

\author{Marijan Markovi\'{c}}
\address{
Faculty of Philosophy\endgraf
University of Montenegro\endgraf
Filozofski fakultet\endgraf
Danila Bojovi\'{c}a bb\endgraf
81400 Nik\v{s}i\'{c}\endgraf
Montenegro}

\email{marijanmmarkovic@gmail.com}

\begin{abstract}
This paper deals with an extremal problem for harmonic functions in the unit ball of $\mathbf{R}^n$.      We are concerned with the
pointwise sharp estimates for the gradient of real--valued bounded harmonic functions. Our main result may be formulated as follows.
The sharp constants in the estimates for the absolute value  of the radial derivative and the modulus of the gradient of a  bounded
harmonic function   coincide near the boundary of the unit ball.  This result partially confirms a conjecture posed by D. Khavinson.
\end{abstract}



\maketitle

\section{Introduction}

\subsection{Gradient  estimates in plane domains}
For $w=(w_1,w_2)$ which belongs to the upper half--plane $\mathbf{R}^2_+=\{y=(y_1,y_2)\in\mathbf{R}^2:y_2>0\}$ the gradient estimate
\begin{equation}\label{INEQ.UHP}
|\nabla U(w)|\le\frac 2\pi \frac 1{w_2} \sup_{y\in\mathbf{R}^2_+}|U(y)|
\end{equation}
is sharp, if one assumes that $U(y)$ is among bounded harmonic functions in $\mathbf{R}^2_+$. Using the conformal transformation  of
the unit disc $\mathbf{B}^2=\{x\in\mathbf{R}^2:|x|<1\}$ onto $\mathbf{R}^2_+$ given by $w = i ({1+z})/({1-z})$, one easily transfers
\eqref{INEQ.UHP} into the following pointwise optimal estimate
\begin{equation}\label{INEQ.UDISC}
|\nabla U(z)|\le \frac 4\pi \frac1{1-|z|^2}\sup_{x\in\mathbf{B}^2}|U(x)|,
\end{equation}
where this time $U(x)$ is a bounded harmonic function in the unit disc, and $z\in\mathbf{B}^2$ is arbitrary.           For the above
inequality we refer to  \cite{C.IUMJ, K.CMB}; see also Chapter 4 in \cite{KM.BOOK.2007}.

The inequality \eqref{INEQ.UDISC} may be viewed as a harmonic analogy of the classical Schwarz lemma for  bounded analytic functions
\begin{equation*}
|f'(z)|\le \frac1{1-|z|^2} \sup_{w\in\mathbf{B}^2} |f(w)|.
\end{equation*}
The  famous Schwarz--Pick lemma improves the preceding inequality        for analytic functions which map  the unit disc onto itself.

As a recent result there is also a Schwarz--Pick type inequality for harmonic functions which send the unit disc $\mathbf{B}^2$ into
the  interval $(-1,1)$. For this result see the paper of D. Kalaj and M. Vuorinen \cite{KV.PAMS}, where the authors    also   give a
counterexample which shows that one cannot expect a Schwarz--Pick type inequality for harmonic mappings of the unit disc onto itself
without additional assumptions. More detailed,  the inequality \eqref{INEQ.UDISC} is improved in a new direction as
\begin{equation}\label{INEQ.KV}
|\nabla U(z)|\le \frac 4\pi \frac{1-U(z)^2}{1-|z|^2}
\end{equation}
for  real--valued harmonic functions bounded by $1$ in $\mathbf{B}^2$. The main ingredients in the proof of \eqref{INEQ.KV} are  the
classical Schwarz--Pick lemma, certain conformal transformations, and the fact that every harmonic function   in a simply--connected
plane domain is a real part of an analytic function.  The extremal functions for \eqref{INEQ.KV} are given in    \cite{M.INDAG}. The
inequality \eqref{INEQ.KV} is equivalent to the following one $d_h (U(z),U(w))\le 4/\pi d_h(z,w)$ for $z,\, w\in \mathbf{B}^2$; here
$d_h$ stands for the hyperbolic distance in  the unit disc $\mathbf{B}^2$. Therefore, real--valued    harmonic functions  bounded by
$1$ are  Lipschitz continuous with respect to the hyperbolic metric with the optimal Lipschitz constant   equal  to $4/\pi$.  On the
other hand, the  inequality \eqref{INEQ.UDISC}  says that a harmonic function $U:\mathbf{B}^2 \rightarrow(-1,1)$        is Lipschitz
continuous regarding the  Euclidean distance $d_e$ and  the hyperbolic distance $d_h$,  respectively, i.e., $d_e(U(z),U(w))\le 4/\pi
d_h (z,w)$. We would like to refer to the F. Colonna paper \cite{C.IUMJ} where the preceding inequality  is considered for  harmonic
mappings of  $\mathbf{B}^2$  onto itself along with the extremal problem
\begin{equation*}
\sup_{z\ne w} \frac {d_e(U(z),U(w))}{d_h (z,w)}=\frac 4\pi.
\end{equation*}
The estimate \eqref{INEQ.UDISC} is also  true in this case  with  the norm of the  differential  of       $U(z)$  on  the  left side.

In the reference \cite{KM.BOOK.2007}     inequalities  similar to   \eqref{INEQ.UHP} and \eqref{INEQ.UDISC}    are considered in the
context of the  so called real--part theorems   for analytic functions.  To make  a connection, let a function $f(z)$ be analytic in
the unit disc,  and  let its real part $\mathrm{Re}f(z)$ be  bounded there.   Since  $|\nabla\mathrm {Re} f(z)| = |f'(z)|$, one  may
rewrite  \eqref{INEQ.UDISC} as  the following   estimate for  the first derivative
\begin{equation*}
|f'(z)|\le \frac 4\pi \frac {1}{1-|z|^2}\sup_{w\in\mathbf{B}^2}|\mathrm{Re} f(w)|,
\end{equation*}
which  is known as  the Lindel\"{o}f   inequality  in the unit disc.     The classical work of A. J. Macintyre and W. W.  Rogosinski
\cite{MR.ACTA} contains in some cases  explicit pointwise sharp estimates    for modulus of  derivatives of analytic  functions with
$\sup_{w\in\mathbf{B}^2}| f(w)|$ on the right  side.

\subsection{The Khavinson  problem}
The  question arises as to what may be done for bounded real--valued harmonic functions in domains in $\mathbf{R}^n$.   Let us first
precise the  notation in high--dimensional settings. For $x=(x_1,\dots,x_{n-1},x_n)\in\mathbf{R}^n$   denote $x'=(x_1,\dots,x_{n-1})
\in\mathbf{R}^{n-1}$.  Let $\mathbf{R}^n_+=\{(x',x_n)\in\mathbf{R}^n:x_n>0\}$ be the upper half--space, and let $\mathbf{B}^n=\{x\in
\mathbf{R}^n:|x|<1\}$ be the unit ball in $\mathbf{R}^n$. We denote by $\{\mathbf {e}_1,\dots,\mathbf{e} _n\}$  the standard base in
$\mathbf{R}^n$.  In the rest of the paper we assume that $n>2$.

Here and henceforth  $\Gamma(z)$ denotes the Gamma function.             Recently G. Kresin and V. Maz'ya \cite{KM.DCDS} proved that
\begin{equation}\label{EST.KM.HS}
|\nabla U (x)|\le
\frac 4{\sqrt{\pi}} \frac {(n-1)^{ \frac {n-1}2 } } { n^{\frac n2} }\frac {\Gamma\left(\frac n2\right)}
{\Gamma\left(\frac {n-1}2 \right)}\frac 1{x_n} \sup_{y\in\mathbf{R}^n_+} |U(y)|,
\end{equation}
where  $U(y)$ is a (real--valued) bounded harmonic function in $\mathbf{R}^n_+$, and $x=(x',x_n)\in\mathbf{R}^n_+$. This  pointwise
optimal gradient estimate,  which  generalizes the half--plane result stated in the beginning   of the  paper,  arises in result of
solution of the Khavinson type problem in the half--space.

In   order to formulate the Khavinson problem and the corresponding conjecture, we introduce the notation we need;  we  will  mostly
follow \cite{KM.JMS.2010, KM.DCDS}.   For every $x$ in the unit ball $\mathbf{B}^n$ (in the  half--space $\mathbf{R}^n_+$) and   for
every $\ell\in\partial\mathbf{B}^n$   let  $\mathcal{C}(x)$  and  $\mathcal{C}(x;\ell)$ denote, respectively, the optimal number for
the gradient estimate
\begin{equation}\label{GRAD.EST}
\left|\nabla  U(x)\right|\le \mathcal{C}(x)\sup_{y}|U(y)|,
\end{equation}
and the optimal number for the gradient  estimate in the direction $\ell$, i.e., the smallest number such that
\begin{equation}\label{GRAD.EST.ELLL}
\left|\left< \nabla U(x), \ell\right> \right|\le \mathcal{C}(x;\ell) \sup_{y}|U(y)|,
\end{equation}
where $U(y)$ is among bounded harmonic functions in $\mathbf{B}^n$ (in  $\mathbf{R}^n_+$). Although we use the same notation for the
gradient  estimates in two different  settings, we believe that this will not cause any confusion.      We will occasionally use the
standard   notation $\partial U(x)/\partial \ell$ for  $\left< \nabla U(x), \ell\right>$.

It is especially important for us the normal direction accompanied to a point in a considered domain.            We use the notation
$\mathbf{n}_x$ for the normal direction at the given point $x$. In the unit ball setting  for $x\ne0$  the direction  $\mathbf{n}_x$
is  the outward unit  vector  orthogonal  to the boundary of  the unit ball  $|x|\mathbf{B}^n$; in this case  we have $\mathbf{n}_x=
x/|x|$, i.e., $\mathbf{n}_x$ coincides  with the  radial  direction. In the half--space  setting the normal direction $\mathbf{n}_x$
is the outward unit vector orthogonal to the boundary of the half--space $(0,x_n)+\mathbf{R}^n_+$. Obviously,         $\mathbf{n}_x=
- \mathbf{e}_n$  for every $x\in\mathbf{R}^n_+$.

Since
\begin{equation*}
\left|\nabla U(x)\right| \  =   \sup_{\ell  \in\partial \mathbf{B}^n} \left|\left<\nabla U(x), \ell\right>\right|,
\end{equation*}
we clearly have
\begin{equation}\label{VAR.PROB}
\mathcal{C}(x)\,  =  \sup_{\ell  \in\partial \mathbf{B}^n}  \mathcal{C}(x;\ell).
\end{equation}
It turned out that the optimisation problem on the right side of \eqref{VAR.PROB} is very difficult, especially   if  one  considers
harmonic functions in the  unit ball. The Khavinson conjecture and its analogue for the half--space  claim  that

\begin{conjecture}[Cf. \cite{KM.JMS.2010, KM.DCDS}]
\begin{equation*}
\mathcal{C}(x) = \mathcal {C} (x;\mathbf{n}_x).
\end{equation*}
\end{conjecture}

We   briefly review the background  of the above conjecture. In 1992, D. Khavinson \cite{K.CMB} found the sharp coefficient  in  the
pointwise estimate for the absolute value of the radial derivative of a bounded harmonic  function in the ball $\mathbf{B}^3$.    He
made a conjecture that the same coefficient should be appear in the stronger         pointwise sharp estimate for the modulus of the
gradient of a bounded harmonic function in $\mathbf{B}^3$.                          In the recent papers  by G. Kresin and V. Maz'ya
\cite{KM.JMS.2010, KM.DCDS} this problem and its analogue in   the half--space were considered in more general aspect. In particular,
in \cite{KM.DCDS} they proved the above     conjecture for the upper half--space, and, as a consequence, found the sharp coefficient
in the inequality \eqref{EST.KM.HS}. For more information about the Khavinson  problem we refer to Chapter  6 in \cite{KM.BOOK.2012}.

The rest of the paper contains two more sections.             In the following one we obtain various integral representations of the
coefficients $\mathcal{C}(x;\ell)$ for the unit ball.  Particularly,  the representation which is given in Theorem \ref{TH.FINAL} is
very important for our  considerations.  By this  theorem we have $\mathcal{C}(x;\ell_1)  =  \mathcal{C}(x;\ell_2)$   if   the angle
between the straight lines $L_{\mathbf{n}_x}$ and $L_{\ell_1}$ is equal to the angle between $L_{\mathbf{n}_x}$ and    $L_{\ell_2}$.
Recall that  the angle between the straight lines   $L_{\ell_1}=\{\lambda\ell_1:\lambda\in\mathbf{R}\}$ and $L_{\ell_2}=\{\mu\ell_2:
\mu\in \mathbf{R}\}$,  determined by the unit vectors  $\ell_1$ and $\ell_2$, respectively,          is $\arccos \left|\left<\ell_1,
\ell_2\right>\right|$. The main aim of this paper is to treat the optimisation problem in \eqref{VAR.PROB} in the unit ball  setting.
This is done in the third section. At this moment we are not able to prove that the above stated conjecture is true,    but we prove
that the statement of the conjecture holds when $x\in\mathbf{B}^n$ is near the boundary of the unit ball,  i.e., if $|x|\thickapprox
1$. Therefore, in the gradient estimate \eqref{GRAD.EST} one can replace   $\mathcal{C}(x)$ with $\mathcal{C}(x;\mathbf{n}_x)$   (as
it is possible in the half--space setting for all $x\in\mathbf{R}^n_+$), if $|x|$ is near $1$. On the other hand, it seems that  the
problem in \eqref{VAR.PROB} is much more harder when $|x|$ is close to $0$.                 The problem is trivial for $x=0$ (Remark
\ref{RE.C.ZERO}).   Moreover, according to our main result, we may conclude that $\mathcal{C}(x)=\mathcal{C}(x;\ell)$ only for $\ell
=  \pm \mathbf{n}_x$, if $|x|\thickapprox 1$.

In the half--space setting it is possible to obtain the factorization $\mathcal {C}(x;\ell) = x_n^{-1}C(\ell)$, where $C(\ell)$ does
not depend on $x=(x',x_n)\in\mathbf{R}^n_+$. This is proved in \cite{KM.DCDS},  where the integral expression representing $C(\ell)$
is also obtained.   The integral representation of $C(\ell)$ also appears  in this paper in  the second section (see     Proposition
\ref{PR.DCDS} there). The main difference between the half--space case and the unit ball case is that in the later   case one cannot
represent $\mathcal {C}(x;\ell)$  in the form which is a number     depending  only on $\ell$ multiplied by the    asymptotic factor
$(1-|x|)^{-1}$,   as it is possible in the first case; this is the content of Lemma \ref{LE.MOEBIUS}        where  it is proved that
$\mathcal{C}(x;\ell )  =  (1-|x|)^{-1}C(x;\ell)$, and   $C(x;\ell)$ is bounded  as  a function   of two variables.

If $U$ is a bounded harmonic function in  $\mathbf{B}^n$, then
\begin{equation*}
V(x) = |x+\mathbf{e}_n/2|^{2-n}U(T(x)),
\end{equation*}
is a such one function in  $\mathbf{R}^n_+$, where $T:\mathbf{R}^n_+\rightarrow \mathbf{B}^n$ is  the M\"{o}bius transform  given by
\begin{equation*}
T(x)  = \frac {|x+\mathbf{e}_n/2|}{|x+\mathbf{e}_n/2|^2} -\mathbf{e}_n.
\end{equation*}
It seems that this transformation does not provide any fruitful connection between the coefficients     $\mathcal{C}(x;\ell)$ in the
two different settings.  However,   it happens that
\begin{equation*}
{C}(\ell') \, =\lim_{\rho  \rightarrow 1^{-}} {C} (\rho\mathbf{e}_1;\ell),
\end{equation*}
where $\ell'$ and $\ell$ are the directions such that the angles between the corresponding  normal directions are equal.  This limit
relation along with the approach   of  G. Kresin and G. Maz'ya  from \cite{KM.DCDS} is crucial for  our  partial  solution to    the
optimisation  problem.

\subsection*{Acknowledgments}
I am thankful to the referees for providing constructive comments and help in improving the quality of this paper.        I am  also
thankful to Professor D. Kalaj for valuable  suggestions on the problem considered in this paper.

\section{Integral  representations  of  $\mathcal{C}(x;\ell)$ and $C(x;\ell)$}

\subsection{Representations as integrals over the unit sphere}
Finding some symmetries and relations between the numbers   $\mathcal{C}(x;\ell)$   itself is very useful in obtaining the  integral
representations of $\mathcal {C}(x;\ell)$.  First  of all we will prove

\begin{lemma}\label{LE.ORT}
For every  $x\in \mathbf{B}^n$  and  $\ell\in \partial \mathbf{B}^{n}$ we have
\begin{equation*}
\mathcal{C}(\mathcal{A} x;  \mathcal{A} \ell)  =  \mathcal{C} (x; \ell),
\end{equation*}
where $\mathcal{A}\in O(n)$.
\end{lemma}

In the above lemma $O(n)$ is the group of all orthogonal transformations of $\mathbf{R}^n$.   It is well known that the group $O(n)$
preserves harmonic functions in the unit ball;  for example, see Chapter 1 in \cite{ABR.BOOK}.

\begin{proof}
Let $\mathcal {A}$ be an orthogonal transformation of $\mathbf{R}^n$. We have
\begin{equation*}\begin{split}
\mathcal{C} (\mathcal {A}x;\mathcal {A} \ell)\
&= \sup_{U,\, \|U\|_\infty=1} \left|\left<\nabla U(\mathcal{A}x),\mathcal{A} \ell\right>\right|
\\&= \sup_{U,\, \|U\|_\infty=1} \left|\left<\mathcal {A} \nabla U(\mathcal{A}x), \ell \right>\right|
 \quad  (\text{since}\, \mathcal{A}^*= \mathcal{A})
\\&= \sup_{U,\, \|U\|_\infty=1} \left|\left<\nabla ( U\circ\mathcal {A})(x), \ell\right>\right|
\\&= \sup_{V,\, \|V\|_\infty=1} \left|\left<\nabla  V(x), \ell \right>\right| =  \mathcal{C}(x; \ell),
\end{split}\end{equation*}
which is the desired relation.   The notation $\|\cdot\|_\infty$ used above stands for the essential supremum of a function  defined
on a measure space.
\end{proof}

\begin{remark}\label{RE.C.ZERO}
If we take $x= 0$ in Lemma \ref{LE.ORT}  we obtain $\mathcal{C} (0;\mathcal{A}\ell)=\mathcal{C}(0;\ell)$,  for every $\mathcal{A}\in
O(n)$ and every $\ell\in\partial\mathbf{B}^n$.  Regarding  the  fact that the group  $O(n)$ acts transitively on spheres with center
in  $0\in\mathbf{R}^n$, this implies  that $C(0;\ell)$ does not  depend on the choice  of a direction.
\end{remark}

The  following lemma isolates the asymptotic factor from      $\mathcal{C}(x;\ell)$, and gives the first integral representation  of
$\mathcal {C} (x;\ell)$  as  an  integral  over the unit sphere $\partial \mathbf{B}^n$.

\begin{lemma}\label{LE.MOEBIUS}
For every  $x\in \mathbf{B}^n$  and  $\ell\in \partial \mathbf{B}^{n}$ we have
\begin{equation*}
\mathcal{C}(x;\ell) = (1 - |x|)^{-1} {C}(x;\ell),
\end{equation*}
where  we have denoted
\begin{equation*}
{C}(x;\ell)=\frac n{1+|x|}  \int_{ \partial \mathbf{B}^{n}}
\left|\left<\eta - \frac {n-2}n x,\ell\right>\right| \left|\eta-x\right|^{2-n} d\sigma(\eta);
\end{equation*}
$d\sigma$ stands for the normalized area measure  on  the unit sphere $\partial \mathbf{B}^n$.
\end{lemma}

\begin{remark}
For  $n=2$ from  this lemma it is not hard to  calculate  that
\begin{equation*}
{C}(x;\ell) =  \frac{4}\pi \frac1{1-|x|}
\end{equation*}
for all $x\in\mathbf{B}^2$ and $\ell\in\partial\mathbf{B}^2$. Therefore,
\begin{equation*}
\mathcal{C}(x) = \mathcal {C}(x;\ell) =  \frac{4}\pi \frac1{1-|x|^2}.
\end{equation*}
This gives the pointwise  sharp gradient  estimate \eqref{INEQ.UDISC}.            For  more plane results we refer to \cite{KM.CAOT}.
\end{remark}

\begin{proof}[Proof of Lemma \ref{LE.MOEBIUS}]
Let $U(y)$ be a bounded harmonic function  in the unit ball $\mathbf{B}^n$.            It  has a radial boundary value $U^*(\zeta) =
\lim_{r\rightarrow 1^-} U(r\zeta)$  for almost every  $\zeta\in \partial\mathbf{B}^{n}$            (we assume the area  measure   on
$\partial\mathbf{B}^{n}$), and $U(y)$ may be represented as the Poisson  integral
\begin{equation*}
U(y) = \mathrm{P}[U^*](y)  =  \int_{\partial \mathbf{B}^n} P(y,\zeta)  U^*(\zeta) d\sigma(\zeta),
\end{equation*}
where
\begin{equation*}
P(y,\zeta) = \frac{1-|y|^2}{|y-\zeta|^n}
\end{equation*}
is the Poisson kernel.  One  says that  $U(y)$ is the Poisson extension of $U^*(\zeta)$. Moreover, there holds      $\|U\|_\infty  =
\| U^*\|_\infty$.  Because of the last relation, $\mathrm{P}$ acts as an isometric isomorphism between the space of bounded harmonic
functions in  $\mathbf{B}^n$ and the space of essentially bounded functions on the unit sphere $\partial \mathbf{B}^n$.      For all
these  facts   we refer to Chapter  6 in \cite{ABR.BOOK}.

For   $x\in\mathbf{B}^n$  and $\ell\in\partial \mathbf{B}^{n}$  we have
\begin{equation}\label{EQ.NABLA.ELL}
\left<\nabla U(x),\ell\right>=\int_{\partial\mathbf{B}^n} \left<\nabla P(x,\zeta),\ell\right>  U^*(\zeta) d\sigma(\zeta).
\end{equation}
Let         $\Lambda_\ell$  denote the functional given on the left side of the above equation.  It operates on the space of bounded
harmonic  functions in the unit ball, and we clearly have $\|\Lambda_\ell\|=\mathcal {C}(x;\ell)$. On the other hand, in view of the
isometric isomorphism via the Poisson extension $\mathrm{P}$ between the space of all bounded harmonic functions in   $\mathbf{B}^n$
and $L^\infty(\partial \mathbf{B}^n)$, the functional $\Lambda_{\ell}$  may be identified with the  functional on $L^\infty(\partial
\mathbf{B}^n)$ given by the right side of the equality  \eqref{EQ.NABLA.ELL}.  Therefore,
\begin{equation}\label{EQ.FIRST.INT}
\mathcal {C}(x;\ell) = \int_{\partial \mathbf{B}^n}  \left|\left<\nabla P(x,\zeta),\ell\right>\right|  d\sigma(\zeta).
\end{equation}

Let $T_x (y)$ be the M\"{o}bius transform
\begin{equation*}
T_x (y) = \frac { (1-|x|^2)(y-x)  - |y-x|^2  x}{[y,x]^2},
\end{equation*}
where
\begin{equation*}
[y,x] = |y||y^*  - x|,\quad  y^* = \frac y{|y|^2}.
\end{equation*}
The mapping  $T_x$ transforms  the unit sphere $\partial\mathbf{B}^n$ onto itself bijectively.    Moreover,   when restricted to the
unit sphere  it  takes the following form
\begin{equation*}
{T}_{x}(\eta) = (1-|x|^2)\frac{\eta-x}{|\eta-x|^2}-x;
\end{equation*}
note that $[\eta,x]=|\eta-x|$ for $\eta\in\partial\mathbf{B}^n$, since $\eta^*=\eta$. We refer to the Ahlfors book \cite{A.BOOK} for
a detailed survey of this class of important  mappings.

We will  take the change of  variable
\begin{equation*}
\zeta = - {T}_{x}(\eta),\quad\eta\in\partial \mathbf{B}^n
\end{equation*}
in the integral representation of $\mathcal{C}(x;\ell)$ given in  \eqref{EQ.FIRST.INT}.           First of all one immediately finds
\begin{equation}\label{NABLA.P}
\nabla P(x,\zeta) = \frac{-2x|x-\zeta|^2 -n(1-|x|^2)(x-\zeta)} {|x-\zeta|^{n+2}}.
\end{equation}
Since
\begin{equation*}
x-\zeta  = (1-|x|^2)\frac{\eta-x}{|\eta-x|^2},  \quad |x-\zeta|=\frac{1-|x|^2}{|\eta-x|},
\end{equation*}
we obtain
\begin{equation*}
\nabla P(x,\zeta)
=\left( {\frac{-2x(1-|x|^2)^2} {|\eta-x|^2} - n\frac{(1-|x|^2)^2(\eta-x)}{|\eta-x|^2}} \right):
 {\frac{(1-|x|^2)^{n+2}}{|\eta-x|^{n+2}}},
\end{equation*}
which implies
\begin{equation*}
\nabla P(x,\zeta) = ({(n-2)x-n\eta}) :  {\frac{(1-|x|^2)^n}{|\eta-x|^n}}.
\end{equation*}
Since
\begin{equation*}
d\sigma(\zeta)     =    \frac{ (1-|x|^2)^{n-1}}{|\eta-x|^{2n-2}}d\sigma(\eta)
\end{equation*}
(see the Ahlfors book), regarding the  above  expression for  $\nabla P(x,\zeta)$, we have
\begin{equation*}\begin{split}
\left<\nabla P(x,\zeta),\ell\right> d\sigma(\zeta)=   -  (1-|x|^2)^{-1}
\left<n\eta - (n-2)x,\ell\right>\left|\eta-x\right|^{2-n}  d\sigma(\eta).
\end{split}\end{equation*}

The preceding facts  imply
\begin{equation*}\begin{split}
\mathcal{C}(x;\ell)\, &=\int_{\partial \mathbf{B}^n}
\left|\left<\nabla P(x,\zeta),\ell\right>\right| d\sigma(\zeta)
\\&=(1-|x|^2)^{-1}
\int_{\partial \mathbf{B}^n} \left|\left<n\eta - (n-2)x, \ell\right>\right|
\left|\eta-x\right|^{2-n} d\sigma(\eta)
\\&=(1-|x|)^{-1} {C}(x;\ell).
\end{split}\end{equation*}
Recall that in the formulation of this lemma we have denoted
\begin{equation*}\begin{split}
{C}(x;\ell)& = \frac n{1+|x|} \int_{\partial \mathbf{B}^n} \left|\left<\eta -\frac{n-2}n x, \ell\right>\right|
\left|\eta-x\right|^{2-n} d\sigma(\eta).
\end{split}\end{equation*}
We have achieved what we wanted to prove.
\end{proof}

\begin{remark}
Since ${C}(x;\ell)=(1-|x|) \mathcal {C}(x;\ell)$ it follows the first lema of this section that
\begin{equation*}
C(\mathcal {A}x;\mathcal{A}\ell) =  C(x;\ell)
\end{equation*}
for all orthogonal transformations $\mathcal{A}$. This also may be seen directly from the just  obtained expression  for $C(x;\ell)$
(similarly as in  the following  lemma).
\end{remark}

\begin{lemma}\label{LE.FIX}
Let $\mathcal{A}\in O(n)$  be a such transformation that   $\mathcal{A}x=x$. Then
\begin{equation*}
{C}(x; \ell) =  {C} (x;\mathcal {A}\ell)
\end{equation*}
for every $\ell \in\partial\mathbf{B} ^n$.
\end{lemma}

\begin{proof}
Since  $\mathcal{A}^{-1}=\mathcal {A}^*$, we also have $\mathcal{A}^* x = x$.  It follows
\begin{equation*}
|\mathcal {A}^*\eta-x| =|\mathcal {A}^*(\eta-x)|=|\eta-x|,
\end{equation*}
and
\begin{equation*}
\left<\mathcal{A}^* \eta-\frac{n-2}n x,\ell\right>=\left<\mathcal {A}^*\left(\eta-\frac {n-2}n x\right),\ell\right>
=\left<\eta-\frac{n-2}n x,\mathcal {A}\ell \right>.
\end{equation*}
In the integral representation of $C(x;\ell)$ in Lemma \ref{LE.MOEBIUS} we will perform  the change of variable: $\mathcal{A}^*\eta$
instead of $\eta$. Since the measure $d\sigma $ is rotation--invariant,   we have
\begin{equation*}\begin{split}
{C}(x;\ell)&
=\frac n{1+|x|} \int_{\partial \mathbf{B}^n}
\left|\left<\mathcal {A}^*\eta-\frac{n-2}n  x, \ell \right>\right|  \left|\mathcal {A}^*\eta - x\right|^{2-n} d\sigma(\eta)
\\&=\frac n{1+|x|} \int_{\partial \mathbf{B}^n}\left|\left< \eta-\frac{n-2}n  x,\mathcal {A} \ell \right>\right|
\left|\eta-x\right|^{2- n}d\sigma(\eta)  = {C}(x;\mathcal {A} \ell),
\end{split}\end{equation*}
which    we  need.
\end{proof}

\begin{lemma}\label{LE.ANGLE}
If the angle between $\mathbf{n}_x$ and $\ell$ is equal to the angle between $\mathbf{n}_x$ and $\ell'$, then
\begin{equation*}
C(x;\ell)=C(x;\ell').
\end{equation*}
\end{lemma}

\begin{proof}
Since   $\left<\mathbf{n}_x,\ell \right> = \left<\mathbf{n}_x,\ell'\right>$, there exists  an orthogonal transformation $\mathcal{A}$
such that $\mathcal {A}\mathbf {n}_x = \mathbf {n}_x$ (which is the same condition as $\mathcal{A}x=x$)   and $\mathcal{A}\ell=\ell'$.
It view  of Lemma \ref{LE.FIX}   the statement of the current one follows.
\end{proof}

\subsection{Representations as double integrals}
Let $\omega_n$ denote the area of the unit sphere in $\mathbf {R}^n$. It is well known that
\begin{equation*}
\omega _n = \frac {2\pi^{\frac n2}}{\Gamma\left(\frac n2\right)}.
\end{equation*}
We set $\omega_1 = 2$, which is consistent with the above formula.

For $\tau\in[0,2\pi]$ let
\begin{equation*}
\ell_\tau =\cos\tau \mathbf{e}_1 +  \sin\tau \mathbf{e}_2.
\end{equation*}
In    this subsection we will restrict our attention to the numbers $C(\rho \mathbf{e}_1;\ell_\tau)$, where $0\le \rho\le1$ and $0\le
\tau\le 2\pi$.

\begin{lemma}\label{LE.DOUBLE.1}
For every $\rho\in [0,1]$ and $\tau\in[0,2\pi]$ we have
\begin{equation*}\begin{split}
{C}(\rho \mathbf {e}_1; \ell_\tau) = \frac{2\omega_{n-2}}{\omega_n}\frac 1{1+\rho}
\int_0^\pi\mathcal{R}(\tilde{\vartheta}) \sin^{n-2} \tilde{\vartheta} d\tilde{\vartheta}
\int_0^\pi |\tilde{\mathcal {G}} (\varphi,\vartheta,\tau )| \sin^{n-3}\varphi d\varphi,
\end{split}\end{equation*}
where
\begin{equation*}
\tilde{\mathcal {G}} (\varphi,\tilde{\vartheta},\tau )
= \frac n2(\cos\tilde{\vartheta}-\alpha_\rho)\cos \tau
+\frac n2(\sin \tilde{\vartheta}\cos \varphi) \sin \tau,\quad \alpha_\rho = \frac {n-2}n\rho,
\end{equation*}
and
\begin{equation*}
\mathcal{R}(\tilde{\vartheta})  =  (1+\rho^2-2\rho\cos \tilde{\vartheta})^{1-\frac n2}.
\end{equation*}
\end{lemma}

\begin{proof}
In the integral  representation of  $C(\rho\mathbf{e}_1;\ell_\tau)$ in Lemma \ref{LE.MOEBIUS}:
\begin{equation*}
{C}(\rho\mathbf{e_1};\ell_\tau)=\frac n{1+\rho}  \int_{ \partial \mathbf{B}^{n}}
\left|\left<\eta-\frac {n-2}n \rho\mathbf{e_1},\ell_\tau\right>\right| \left|\eta-\rho\mathbf{e}_1\right|^{2-n} d\sigma(\eta)
\end{equation*}
we will take the change of variable introducing the spherical coordinates in the following way
$\eta=(\cos\tilde{\vartheta},\sin\tilde{\vartheta}\cos\varphi,\dots)$. The range of $\tilde{\vartheta}$ is $[0,\pi]$.   If $n>3$, the
range of $\varphi$ is $[0,\pi]$; if  $n=3$, then the range   of  $\varphi$ is  $[0,2\pi]$.

Straightforward calculations  yield
\begin{equation*}\begin{split}
\left|\left<\eta-\alpha_\rho  \mathbf {e}_1,\ell_\tau\right>\right|
=  |(\cos \tilde{\vartheta}-\alpha_\rho)\cos \tau + (\sin \tilde{\vartheta}\cos \varphi) \sin \tau |
\end{split}\end{equation*}
and
\begin{equation*}\begin{split}
|\eta-\rho \mathbf {e}_1 |^{2-n} = (1+\rho^2-2\rho\cos \tilde{\vartheta})^{1- \frac n2}.
\end{split}\end{equation*}

Assume first that $n>3$. Then we have
\begin{equation*}\begin{split}
{C}(\rho \mathbf {e}_1; \ell_\tau)
&= \frac n{1+\rho}
\int_{\partial \mathbf{B}^n} \left|\left<\eta-\alpha _\rho \mathbf {e}_1,\ell_\tau\right>\right|
\left|\eta-\rho \mathbf {e}_1\right|^{2-n} d\sigma(\eta)
\\&=\frac{2\omega_{n-2}}{\omega_n}\frac 1{1+\rho }
\int_0^\pi\int_0^\pi |\tilde{\mathcal{G}}(\varphi,\tilde{\vartheta},\tau)  | \mathcal{R}(\tilde{\vartheta})
\sin^{n-2} \tilde{\vartheta} \sin^{n-3}\varphi d\tilde{\vartheta}d\varphi
\\&=\frac{2\omega_{n-2}}{\omega_n}\frac 1{1+\rho}
\int_0^\pi\mathcal{R}(\tilde{\vartheta}) \sin^{n-2} \tilde{\vartheta} d\tilde{\vartheta}
\int_0^\pi |\tilde{\mathcal {G}} (\varphi,\vartheta,\tau )| \sin^{n-3}\varphi d\varphi.
\end{split}\end{equation*}

The   obtained expression for ${C}(\rho \mathbf {e}_1; \ell_\tau)$ is also valid if $n=3$,  but  in this case we must bear in mind one
more   transformation of the inner integral above where the integration is against $\varphi$.                             This step is
\begin{equation*}\begin{split}
\int_\pi^{2\pi} |\tilde{\mathcal{G}}(\varphi,\tilde{\vartheta},\tau)  |  d\varphi
&=\int_0^{\pi} |\tilde{\mathcal{G}}(\pi +\varphi,\tilde{\vartheta},\tau)  | d\varphi
\\&=\int_0^{\pi} |\tilde{\mathcal{G}}(\pi -\varphi,\tilde{\vartheta},\tau)  | d\varphi
\\&=\int_0^{\pi} |\tilde{\mathcal{G}}(\varphi,\tilde{\vartheta},\tau)  | d\varphi
\end{split}\end{equation*}
(we have introduced   $\varphi$   instead of $\pi -\varphi$  in the last integral).
\end{proof}

Using   some   trigonometric  transformations  we will now prove

\begin{lemma}\label{LE.DOUBLE.2}
For every $\rho\in [0,1]$ and $\tau\in[0,2\pi]$ we have
\begin{equation*}\begin{split}
{C}(\rho \mathbf{e}_1;\ell_\tau)= \frac{4\omega_{n-2}}{\omega_n}\frac1{1+\rho}
\int_0^\pi\sin^{n-3}\varphi d\varphi
\int_0^{\frac \pi2} |\mathcal {G}(\varphi,\vartheta,\tau) | \mathcal{S}(\vartheta) d\vartheta,
\end{split}\end{equation*}
where   we have denoted
\begin{equation*}
\mathcal {G}(\varphi,\vartheta,\tau)
= (n\cos ^2\vartheta - \beta_\rho)\cos\tau  + (n\sin \vartheta \cos\vartheta\cos \varphi)\sin\tau,
\quad \beta_\rho=\frac{n-(n-2)\rho}{2},
\end{equation*}
and
\begin{equation*}
\mathcal{S}(\vartheta)   =  \frac {\sin^{n-2}2\vartheta}{ ((1+\rho)^2-4\rho\sin^2 \vartheta)^{ \frac n2-1}}.
\end{equation*}
\end{lemma}

\begin{proof}
Interchanging the order of integration, from the preceding lemma we immediately  obtain
\begin{equation*}\begin{split}
{C}(\rho \mathbf{e}_1; \ell_\tau)=
\frac{2 \omega_{n-2}}{\omega_n}\frac1{1+\rho}
\int_0^\pi \sin^{n-3}\varphi d\varphi \int_0^\pi |\tilde{\mathcal {G}} (\varphi,\tilde{\vartheta},\gamma)|
\mathcal{R}(\tilde{\vartheta})\sin^{n-2} \tilde{\vartheta}  d\tilde{\vartheta},
\end{split}\end{equation*}
where
\begin{equation*}
\tilde{\mathcal {G}} (\varphi,\tilde{\vartheta},\tau)
= \frac n2(\cos\tilde{ \vartheta}-\alpha_\rho)\cos \tau +  \frac n2 (\sin\tilde{ \vartheta}\cos \varphi)\sin\tau,
\end{equation*}
and
\begin{equation*}
\mathcal{R}(\tilde{\vartheta}) =  (1+\rho^2-2\rho \cos \tilde{\vartheta})^{1-\frac n2}.
\end{equation*}

Introduce $\vartheta = {\tilde{\vartheta}}/2$.  We transform
\begin{equation*}\begin{split}
|\tilde{\mathcal {G}} (\varphi,\tilde{\vartheta},\tau )|
&  =  \frac n2|(\cos \tilde{\vartheta}-\alpha_\rho)\cos\tau  +  (\sin \tilde{\vartheta}\cos \varphi) \sin\tau|
\\&=  \frac n2|(1-2\sin^2 \vartheta-\alpha_\rho)\cos\tau +  (2 \sin \vartheta\cos\vartheta\cos \varphi)\sin\tau |
\\&= | (n \sin^2 \vartheta-  n(1-\alpha_\rho)/2)\cos\tau   -   (n \sin \vartheta\cos\vartheta\cos \varphi)\sin\tau |
\\&= | (n \sin^2 \vartheta-\beta_\rho ) \cos\tau-   (n \sin \vartheta\cos\vartheta\cos \varphi)\sin\tau|.
\end{split}\end{equation*}
In the last equality we have used the substitution
\begin{equation*}
\frac {n(1-\alpha_\rho)}2=  \frac{n-(n-2)\rho}{2}=\beta_\rho.
\end{equation*}

In   the  integral representation of $C(\rho\mathbf{e}_1;\ell_\tau)$ given above we will write $\pi-\varphi$ instead of $\varphi$, and
$\pi/2-\vartheta$ instead of $\vartheta$ (i.e. ${\tilde{\vartheta}}/2$).  In this way we obtain
\begin{equation*}\begin{split}
{C}(\rho\mathbf{e}_1;\ell_\tau)&=\frac{4\omega_{n-2}}{ \omega_n}\frac1{1+\rho}    \int_0^\pi\sin^{n-3}\varphi d\varphi
\int_0^{\frac \pi2} |\mathcal {G}(\varphi,\vartheta,\tau) |\mathcal{S}(\vartheta) d\vartheta,
\end{split}\end{equation*}
where
\begin{equation*}\begin{split}
\mathcal {G}(\varphi,\vartheta,\tau) &
=     (n \sin^2(\pi/2 - \vartheta)-\beta_\rho )\cos\tau  -  (n \sin \vartheta\cos \vartheta\cos(\pi- \varphi))\sin\tau
\\&= (n\cos ^2\vartheta - \beta_\rho)\cos\tau  +  (n \sin \vartheta \cos\vartheta\cos \varphi)\sin\tau,
\end{split}\end{equation*}
and
\begin{equation*}\begin{split}
\mathcal{S}(\vartheta) &=\sin^{n-2} 2( \pi/2- \vartheta)\mathcal{R}(2(\pi/2-\vartheta))
=\frac {\sin^{n-2} 2{\vartheta}} {( 1+\rho^2+2\rho\cos2{\vartheta})^{ \frac n2-1}}
\\& =\frac{ \sin^{n-2}2\vartheta}{ ((1+\rho)^2-4\rho\sin^2 \vartheta)^{ \frac n2-1}}.
\end{split}\end{equation*}
This proves our  lemma.
\end{proof}

Now, we can prove  the following symmetry.

\begin{lemma}\label{LE.SYMMETRIC}
If $\tau\in [0,\pi]$, then
\begin{equation*}
{C}(\rho \mathbf {e}_1; \ell_{\pi-\tau})  =  {C}(\rho \mathbf {e}_1; \ell_\tau)
\end{equation*}
for all $\rho \in [0,1]$.
\end{lemma}

\begin{proof}
One obtains this relation immediately from the just derived double integral expression for $C(\rho\mathbf{e}_1;\ell_\tau)$.  Introduce
there $\pi-\tau$, and then the change of variable:                $\pi-\varphi$ instead of $\varphi$. Particularly, for the expression
$\tilde{\mathcal {G}} (\varphi,\tilde{\vartheta},\tau )$  we have
\begin{equation*}\begin{split}
|\tilde{\mathcal {G}} (\pi-\varphi,\tilde{\vartheta},\pi-\tau )|&
=\frac n2|(\cos\tilde{\vartheta}-\alpha_\rho)\cos (\pi-\tau) + \sin \tilde{\vartheta}\cos (\pi -\varphi) \sin (\pi-\tau)|
\\& = \frac  n2|-(\cos\tilde{\vartheta}-\alpha_\rho)\cos \tau -\sin \tilde{\vartheta}\cos  \varphi  \sin \tau|
= |\tilde{\mathcal {G}} (\varphi,\tilde{\vartheta},\tau )|.
\end{split}\end{equation*}
It follows  the statement  of this lemma.
\end{proof}

\subsection{An integral equal to zero}
Our next aim is to show that if we delete the parentheses for absolute value in the integral representation of   $C(\rho\mathbf{e}_1;
\ell_\tau)$   given in Lemma \ref{LE.DOUBLE.2}, then the double integral obtained   in this way is equal to zero.         This  is  a
crucial  fact  in obtaining the final integral   representations of the coefficients  $C(x;\ell)$     in the next subsection.      To
prove  that we use  some facts concerning the Gauss hypergeometric function. For these facts  which  will be  stated  below  we refer
to Chapter 2    and Chapter 3  in \cite{AAR.BOOK}.

The Gauss hypergeometric function is defined by the hypergeometric series
\begin{equation*}
{}_2 F_1\left(
\begin{matrix}
a\hspace{0.2cm} b \hspace{0.1cm}\\
c
\end{matrix}
;z\right) =    \sum _{k=0}^\infty \frac {(a)_k (b)_k}{(c)_k}\frac {z^k}{k!}
\end{equation*}
for $z\in\mathbf{B}^2$,   and by continuation elsewhere.          Here, $(a)_k$ stands for  the Pochhammer symbol which is defined as
\begin{equation*}
(a)_k = \left\{
\begin{array}{ll}
a(a + 1)\cdots (a + k  - 1),   & \hbox{if $k\ge1$,} \\
1,                             & \hbox{if $k=0$}
\end{array}
\right.
\end{equation*}
for every  complex number $a$.  The series terminates if either $a$ or $b$ is a non--positive integer,    in which  case the function
reduces to a polynomial. For example,  if  $b=-m$, then
\begin{equation*}
{}_2 F_1\left(
\begin{matrix}
a \hspace{0.2cm} -m \hspace{0.1cm}\\
c
\end{matrix}
;z\right) =     \sum _{k=0}^m  (-1)^k\binom{ m}{k} \frac {(a)_k  }{(c) _k}  {z^k}.
\end{equation*}

One of the fundamental relations is the Euler integral representation formula
\begin{equation*}
{}_2 F_1\left(
\begin{matrix}
a \hspace{0.2cm} b\hspace{0.1cm} \\
c
\end{matrix}
;z\right)
 =\frac {\Gamma(c)}{\Gamma(b)\Gamma(c-b)}    \int_0^1t^{b-1}(1-t)^{c-b-1}(1-zt)^{-a}dt
\end{equation*}
valid  for $\Re{c}>\Re{b}>0$ and $z\ne 1,\, |\mathrm{arg}(1-z)|<\pi$.

One form of the quadratic transform  formula  states that
\begin{equation*}
{}_2 F_1\left(
\begin{matrix}
a/2 \hspace{0.2cm} (a+1)/2 \hspace{0.1cm}\\
 a-b+1
\end{matrix}
;z\right) =\left( \frac {1+\sqrt{1-z}}2\right)^{- a}{}_2 F_1\left(
\begin{matrix}
a \hspace {0.2cm}  b\\
 a-b+1\hspace{0.1cm}
\end{matrix}
;{\frac{1-\sqrt{1-z}}{1+\sqrt{1-z}}}\right)
\end{equation*}
for  $z\in\mathbf{B}^2$.

\begin{lemma}\label{LE.ZERO}
The equality
\begin{equation*}
\int_0^\pi\sin^{n-3}\varphi d\varphi\int_0^{\frac \pi2} \mathcal {G} (\varphi,\vartheta,\tau)
 \mathcal{S}(\vartheta) d\vartheta=0
\end{equation*}
is valid for every  $\tau \in[0,2\pi]$.
\end{lemma}

\begin{proof}
We have to  prove that
\begin{equation*}
\int_0^\pi\sin^{n-3}\varphi d\varphi
\int_0^{\frac \pi2} ((n\cos^2\vartheta-\beta_\rho)\cos\tau  + (n \sin \vartheta \cos\vartheta\cos \varphi)\sin\tau)
\mathcal{S}(\vartheta) d\vartheta=0.
\end{equation*}

Since
\begin{equation*}\begin{split}
\int_0^\pi\sin^{n-3}\varphi \cos \varphi d\varphi=0,
\end{split}\end{equation*}
we have
\begin{equation*}
\int_0^\pi\sin^{n-3}\varphi d\varphi
\int_0^{\frac\pi2} (n   \sin \vartheta \cos\vartheta\cos \varphi )  \mathcal{S}(\vartheta) d\vartheta=0.
\end{equation*}
It  remains to consider the integral expression
\begin{equation*}\begin{split}
\int_0^\pi\sin^{n-3}\varphi d\varphi
\int_0^{\frac \pi2} (n\cos^2 \vartheta-\beta_\rho)\mathcal{S}(\vartheta) d\vartheta
&= \frac { \sqrt{\pi} \Gamma\left( \frac {n-2}2 \right)} {\Gamma\left(\frac {n-1}2\right)}  J,
\end{split}\end{equation*}
where
\begin{equation*}
{J}= \int_0^{\frac \pi2} (n\cos^2 \vartheta-\beta_\rho)\mathcal{S}(\vartheta) d\vartheta.
\end{equation*}
We will  prove that $  {J} =0$, i.e., that $  {J}_1 = {J}_2$, where
\begin{equation*}
{J}_1= n \int_0^{\frac \pi2} \cos^2 \vartheta\mathcal{S}(\vartheta) d\vartheta,\quad
{J}_2=\beta_\rho \int_0^{\frac \pi2}\mathcal{S}(\vartheta) d\vartheta.
\end{equation*}

In the  integral  $J_1$ introduce the change of variable
\begin{equation*}
\vartheta = {\arcsin}\sqrt{t},\quad d\vartheta = \frac 12 t^{- \frac 12}(1-t)^{- \frac 12}dt.
\end{equation*}
We obtain
\begin{equation*}\begin{split}
J_1  &=
n \int_0^{\frac \pi2} \cos^2 \vartheta\mathcal{S}(\vartheta) d\vartheta = 2^{n-2} n \int_0^{\frac\pi2}
\frac{  \sin^{n-2}\vartheta\cos^{n}\vartheta}{ ((1+\rho)^2-4\rho\sin^2 \vartheta)^{\frac n2-1}}d\vartheta
\\&=2^{n-3} n \int_0^{1} \frac{t^{ \frac {n-3}2 }(1-t)^{ \frac {n-1}2}}{ ((1+\rho)^2-4\rho t)^{ \frac n2-1}}dt
\end{split}\end{equation*}
Let
\begin{equation*}
\tilde{\rho}=  \frac{4\rho}{(1+\rho)^2}.
\end{equation*}
Note that
\begin{equation*}
\sqrt{1-\tilde{\rho}}= \frac {1-\rho}{1+\rho}.
\end{equation*}
By the  Euler  integral  formula it follows
\begin{equation*}\begin{split}
J_1 &=\frac{2^{n-3} n}{ (1+\rho)^{n-2}}
 \int_0^{1} t^{ \frac {n-3}2}(1-t)^{\frac {n-1}2} (1-\tilde{\rho}t)^{1- \frac n2}dt
\\& =\frac{2^{n-3} n}{ (1+\rho)^{n-2}} \frac{\Gamma\left( \frac {n-1} 2\right)
\Gamma\left( \frac {n+1}2\right)}{\Gamma\left(n\right)}
{}_2 F_1\left(
\begin{matrix}
 n/2-1  \hspace{0.2cm} (n-1)/2  \hspace{0.1cm}\\
n
\end{matrix} ;\tilde{\rho}\right).
\end{split}\end{equation*}
Using now the quadratic transform formula    we obtain
\begin{equation*}\begin{split}
J_1 &=\frac{2^{n-3} n}{ (1+\rho)^{n-2}}
\frac{\Gamma\left(\frac {n-1}2 \right)\Gamma\left(\frac {n+1}2\right)}{\Gamma\left(n\right)} (1+\rho)^{n-2}
 {}_2 F_1\left(
\begin{matrix}
 n-2 \hspace{0.2cm} -1 \hspace{0.1cm}\\
n
\end{matrix}
;\rho \right)
\\& =  {2^{n-3} n} \frac {n-1}2\frac{\Gamma^2\left(\frac {n-1}2 \right)}{(n-1)\Gamma\left(n-1\right)}\frac {n-(n-2)\rho}{n}
\\& =  {2^{n-3}  } \frac{ \Gamma^2\left(\frac {n-1}2\right)}{ \Gamma\left(n-1\right)}\beta_\rho.
\end{split}\end{equation*}

Calculation of $J_2$ goes in the same way. We have
\begin{equation*}\begin{split}
J_2&= \beta_\rho \int_0^{\frac \pi2}\mathcal{S}(\vartheta) d\vartheta
= 2^{n-2}\beta_\rho \int_0^{\frac \pi2}
\frac{\sin^{n-2}\vartheta\cos^{n-2}\vartheta}{ ((1+\rho)^2-4\rho\sin^2 \vartheta)^{\frac n2-1}}d\vartheta
\\&=  2^{n-3}\beta_\rho \int_0^{1} \frac{t^{ \frac {n-3}2}(1-t)^{\frac {n-3}2}}{ ((1+\rho)^2-4\rho t)^{\frac  n2-1}}dt.
\end{split}\end{equation*}
Using the Euler integral formula and then  the quadratic transform formula   we calculate
\begin{equation*}\begin{split}
J_2 &= \frac{2^{n-3}\beta_\rho} {(1+\rho)^{n-2}}
 \int_0^{1} t^{\frac {n-3}2}(1-t)^{ \frac {n-3}2} (1-\tilde{\rho}t)^{1-\frac n2}dt
\\& =\frac{2^{n-3} \beta_\rho}{ (1+\rho)^{n-2}}
 \frac{\Gamma^2\left( \frac{n-1}2\right)}{\Gamma\left(n-1\right)}
{}_2 F_1\left(
\begin{matrix}
n/2-1\hspace{0.2cm} (n-1)/2 \hspace{0.1cm}\\
n-1
\end{matrix}
;\tilde{\rho}\right)
\\& =\frac{2^{n-3} \beta_\rho}{ (1+\rho)^{n-2}}
 \frac{\Gamma^2\left( \frac{n-1}2\right)}{\Gamma\left(n-1\right)} (1+\rho)^{n-2}
{}_2 F_1\left(
\begin{matrix}
n-2\hspace{0.2cm} 0 \hspace{0.1cm}\\
n-1
\end{matrix}
;\rho\right)
\\& = {2^{n-3} \beta_\rho}
 \frac{\Gamma^2\left( \frac{n-1}2\right)}{\Gamma\left(n-1\right)}.
\end{split}\end{equation*}

Therefore,  ${J}_1 ={J}_2$ which  we aimed  to prove.
\end{proof}

\subsection{The main representation theorem}
Based on the preceding auxiliary  results our aim  now is to find representation of $C(x;\ell )$ which is convenient for the extremal
problem consideration. This  is contained in the next theorem which is the main result of this section. We specify gradient estimates
in  the  normal and  tangential directions.

Recall that  we have already  introduced the following two  parameters
\begin{equation*}
\alpha_\rho = \frac {(n-2)\rho}{n},\quad \beta_\rho=\frac  {n-(n-2)\rho}{2}.
\end{equation*}

\begin{theorem}\label{TH.FINAL}
Let   $x\in\mathbf{B}^n$ and $\ell\in\partial \mathbf{B}^n$.    Denote by $\tau$ the angle between the straight  lines  $N =\{\lambda
\mathbf{n}_x:\lambda\in\mathbf{R}\}$ and  $L = \{\mu\ell:\mu\in\mathbf{R}\}$. The optimal coefficient $C(x;\ell)$ for  the   estimate
\begin{equation*}
\left|\frac {\partial U(x)}{\partial \ell}\right|\le C(x;\ell)  (1-|x|)^{-1} \sup_{y\in\mathbf{B}^n}|U(y)|,
\end{equation*}
where  $U(y)$  is  among   bounded  harmonic functions  in $\mathbf{B}^n$, may be expressed as follows.

i) If $0\le \tau<\pi/2$,  then
\begin{equation*}\begin{split}
C(x;\ell)& = \frac{4\omega_{n-2}}{\omega_n}\frac{2^{n-1}}{(1+|x|)^{n-1}}\frac 1{\sqrt{1+\gamma^2}}
\int_{0}^1 \frac{ \mathcal{P}_{|x|}(\gamma t) +\mathcal{P}_{|x|}(-\gamma t)  }{\sqrt{(1-t^2)^{4-n}}}dt,\quad \gamma = \tan \tau,
\end{split}\end{equation*}
where the  function  $\mathcal{P}_\rho(z)$ of a real  argument is defined as
\begin{equation*}
\mathcal {P}_\rho   (z) =  \int_0^{\frac {z+\sqrt{z^2 +1-\alpha_\rho^2}}{1-\alpha_\rho}}
\frac {  (n-\beta_\rho   +  n z w- \beta_\rho w^2)w^{n-2}dw}
{(1+w^2)^{\frac n2+1}(1+\kappa_\rho^2 w^2)^{ \frac  n2-1}},\quad \kappa_\rho = \frac {1-\rho}{1+\rho}
\end{equation*}
for $0\le \rho\le 1$.

ii) The coefficient for the pointwise  sharp gradient   estimate in the radial  direction is given by
\begin{equation*}
C(x;\mathbf{n}_x)  =  \frac{4} {\sqrt{\pi}}\frac{\Gamma\left(\frac n2\right)}
{ \Gamma\left( \frac {n-1}2\right) }\frac{2^{n-1}}{(1+|x|)^{n-1}}
\int_0^{w_{|x|}}\frac {\beta_{|x|}( w_{|x|}^2-w^2) w^{n-2}dw}
{(1+w^2)^{ \frac n2+1} (1+\kappa_{|x|}^2w^2)^{ \frac n2-1}},
\end{equation*}
where $w_\rho = \sqrt{\frac{n+(n-2)\rho}{n-(n-2)\rho}}$.

iii) Let $\mathbf{t}_x$ be  a tangential  direction, i.e., any direction orthogonal to  $\mathbf{n}_x$. Then
\begin{equation*}
C(x;\mathbf{t}_x) =\frac {2n}\pi \frac {2^{n-1}}{(1+|x|)^{n-1}}
\int_0^\infty  \frac {w^{n-1}dw} { (1+w^2)^{ \frac n2+1}(1+\kappa_{|x|}^2w^2)^{\frac n2-1}}.
\end{equation*}
\end{theorem}

\begin{remark}
Other representations of the coefficient in the pointwise sharp estimate for the absolute value of the radial derivative for a bounded
harmonic function in the unit ball were obtained in \cite{KM.JMS.2010} (see  Proposition 4.1 and Corollary  5.4).
\end{remark}

\begin{remark}
For $x=0$ we have $w_0=1$,  $\beta_0 =  n-2$, and $\kappa_0 =1$. One finds
\begin{equation*}\begin{split}
\int_0^{w_0}\frac { \beta_0(w_0^2    -  w^2) w^{n-2}dw}
{(1+w^2)^{ \frac n2+1} (1+\kappa_0^2w^2)^{ \frac n2-1}}&
=\frac n2\int_0^{1}\frac {(1-w^2)w^{n-2}} {(1+w^2)^{ n} }dw\\&=\frac 1 {2^{n}} \frac n{n-1}.
\end{split}\end{equation*}
Since $C(0;\ell)$ is independent of $\ell\in\partial\mathbf{B}^n$ (Remark \ref{RE.C.ZERO}),  from the part ii) of our theorem  we have
the following  explicit sharp gradient  estimate at zero
\begin{equation}\label{EST.ZERO}
|\nabla U (0)|\le \frac 2{\sqrt{\pi}}\frac{ \Gamma\left(\frac {n+2}2\right)}
{ \Gamma\left(\frac {n+1}2\right)}\sup_{y\in\mathbf{B}^n}|U(y)|.
\end{equation}
This estimate is well known,     it may be found, for example, in Theorem 6.2.6 in  \cite{ABR.BOOK}, or in a  more general setting  in
Corollary 3.2 in  \cite{KM.JMS.2010}.
\end{remark}

\begin{remark}
The above gradient estimate in zero may be used to obtain the smallest constant     $C$ (independent of $x\in\mathbf{B}^n$)  such that
\begin{equation*}
|\nabla U(x)|\le C (1-|x|)^{-1} \sup_{y\in\mathbf{B}^n}|U(y)|.
\end{equation*}
It is clear that $C =\sup_{x\in \mathbf{B}^n,\, \ell\in \partial\mathbf{B}^n} C(x;\ell)$, but is seems that it is not an  easy task to
find this extremum. However, we can act as follows. Let $U(y)$ be bounded harmonic function in $\mathbf{B}^n$. For $x\in \mathbf{B}^n$
consider  the new harmonic   function $V(y)=U(x+(1-|x|)y)$ for $y\in\mathbf{B}^n$. Then
\begin{equation*}\begin{split}
|\nabla V(0) |  = (1-|x|)|\nabla U(x)|.
\end{split}\end{equation*}
Since $\sup_{y\in\mathbf{B}^n}|V(y)| \le \sup_{x\in\mathbf{B}^n} |U(x)|$,     by        applying \eqref{EST.ZERO} for $V(y)$ we obtain
\begin{equation}\label{INEQ.GLOB}
|\nabla U(x)|\le\frac 2{\sqrt{\pi}}\frac{ \Gamma\left(\frac {n+2}2\right) }
{ \Gamma\left( \frac {n+1}2\right)} \frac 1{1-|x|}\sup_{y\in\mathbf{B}^n}|U(y)|.
\end{equation}
Clearly, this estimate is not pointwise sharp, but it is sharp if we take into consideration  the whole domain.          The  estimate
\eqref{INEQ.GLOB}  may be found in \cite{PW.BOOK} on page 139. For  such type inequalities   we refer to  the paper \cite{KM.JMS.2014}
where  similar optimal inequalities are  considered  for  more  general  domains.
\end{remark}

We will use the following simple observation concerning the integral of a measurable function. Assume that $\phi(x)$ is a real--valued
integrable function on a measure space $(X,\mu)$. Let $\phi^+(x)=\max\{\phi(x),0\}$, and $\phi^{-}(x)   =  \max\{-\phi(x),0\}$. Denote
$\int_X \phi(x) d\mu(x)  =  J$. Since  $\int_X (\phi^+(x) - \phi^-(x)) d\mu(x)  =  J$,           we have  $\int_X \phi^-(x) d\mu(x)  =
\int_X \phi^+(x) d\mu(x) -J$, which  implies that
\begin{equation*}
\int_X |\phi(x)| d\mu(x) = \int_X ( \phi^+(x)+ \phi^-(x)) d\mu(x)  = 2\int_X  \phi^+(x)  d\mu(x) -J.
\end{equation*}
We need  the above equality for $J = 0$.

\begin{proof}[Proof of Theorem \ref{TH.FINAL}]
Since  $C(x;\ell) = C(x;-\ell)$,     regarding Lemma \ref{LE.ORT},  Lemma \ref{LE.ANGLE}, and Lemma  \ref{LE.SYMMETRIC}, one sees that
$C(x;\ell)$, as a function of two variables,                    depends only on $|x|$ and the angle $\tau$ between the straight  lines
$N=\{\lambda\mathbf{n}_x:\lambda\in\mathbf{R}\}$ and $L= \{\mu\ell:\mu\in\mathbf{R}\}$. Therefore,     $C(x;\ell) = C(|x|\mathbf{e}_1;
\ell_\tau)$; recall that $\ell_\tau = \cos\tau \mathbf{e}_1 + \sin\tau \mathbf{e}_2$. It follows that in the sequel we should consider
the numbers $C(\rho\mathbf{e}_1;\ell_\tau)$ for $0\le \rho\le 1$ and $0\le \tau\le\pi/2$.         Therefore, we will continue from the
integral  expression  for  $C(\rho\mathbf{e}_1;\ell_\tau)$ contained in Lemma \ref{LE.DOUBLE.2}.   Recall that   we obtained there the
following
\begin{equation*}\begin{split}
{C}(\rho \mathbf{e}_1;\ell_\tau)= \frac{4\omega_{n-2}}{\omega_n}\frac1{1+\rho}
\int_0^\pi\sin^{n-3}\varphi d\varphi
\int_0^{\frac \pi2} |\mathcal {G}(\varphi,\vartheta,\tau) | \mathcal{S}(\vartheta) d\vartheta,
\end{split}\end{equation*}
where
\begin{equation*}
\mathcal {G}(\varphi,\vartheta,\tau)
= (n\cos ^2\vartheta - \beta_\rho)\cos\tau  + (n\sin \vartheta \cos\vartheta\cos \varphi)\sin\tau,
\end{equation*}
and
\begin{equation*}
\mathcal{S}(\vartheta) = \frac {\sin^{n-2}2\vartheta}{ ((1+\rho)^2-4\rho\sin^2 \vartheta)^{ \frac n2-1}}.
\end{equation*}

In the above integral expression we introduce the change of variables
\begin{equation*}
\vartheta = \arctan w\quad\text{and}\quad \varphi=\arccos t.
\end{equation*}
Then we have
\begin{equation*}
\sin \vartheta = \frac {w}{\sqrt{1+w^2}},\quad \cos \vartheta = \frac {1}{\sqrt{1+w^2}},\quad d\vartheta  = \frac{dw}{1+w^2}.
\end{equation*}
It follows
\begin{equation*}\begin{split}
\mathcal {G}(\varphi,\vartheta,\tau) &
= (n\cos ^2\vartheta - \beta_\rho)\cos\tau + (n  \sin \vartheta \cos\vartheta\cos \varphi)\sin\tau
\\&=\frac{(n-\beta_\rho)\cos\tau + (n t \sin\tau)w - (\beta_\rho\cos\tau )w^2}{1+w^2},
\end{split}\end{equation*}
and after  very short calculation
\begin{equation*}\begin{split}
\mathcal{S}(\vartheta) = \frac {2^{n-2}\sin^{n-2}\vartheta\cos^{n-2}\vartheta}{((1+\rho)^2 - 4\rho\sin^2\vartheta)^{\frac n2-1}}
=\frac{2^{n-2}}{(1+\rho)^{n-2}}\frac {w^{n-2}}{(1+w^2)^{\frac n2-1}} \frac1  { (1+\kappa_\rho^2w^2)^{\frac n2-1}}.
\end{split}\end{equation*}

If $0\le  \tau< \pi/2$, it is convenient to introduce a new parameter $\gamma = \tan \tau$. Then
\begin{equation*}
\cos\tau = \frac {1}{\sqrt{1+\gamma^2}},\quad \sin\tau =\frac \gamma{\sqrt{1+\gamma^2}}.
\end{equation*}
Therefore,
\begin{equation*}
\mathcal {G}(\varphi,\vartheta,\tau)
 = \frac {1}{\sqrt{1+\gamma^2}} \frac{n-\beta_\rho   +  n \gamma t w - \beta_\rho w^2}{1+w^2}.
\end{equation*}
Altogether we have
\begin{equation*}\begin{split}
\mathcal {G}(\varphi,\vartheta,\tau)\mathcal{S}(\vartheta)d\vartheta &
=\frac{2^{n-2}}{(1+\rho)^{n-2}}  \frac {   (n-\beta_\rho  +  n \gamma t w-\beta_\rho w^2)w^{n-2}  }
{ (1+w^2)^{ \frac n2+1}(1+\kappa_\rho^2 w^2)^{\frac n2-1}}dw
\\& =\frac{2^{n-2}}{(1+\rho)^{n-2}}  Q(w) W(w)dw.
\end{split}\end{equation*}
The unique positive zero of the quadratic expression
\begin{equation*}
Q(w) =  n-\beta_\rho  +n z w -\beta_\rho w^2
\end{equation*}
is
\begin{equation*}\begin{split}
Z(z) =\frac{z + \sqrt{z^2 +1-\alpha^2_\rho}}{1-\alpha_\rho}.
\end{split}\end{equation*}
Regarding the last expression for      $C(\rho\mathbf{e}_1;\ell_\tau)$ stated in the beginning of this proof and the observation given
before this proof we have
\begin{equation*}\begin{split}
C(\rho\mathbf{e}_1;\ell_\tau) & =\frac{4\omega_{n-2}}{\omega_n}\frac {2^{n-2}}{(1+\rho)^{n-1}}\frac 1{\sqrt{1+\gamma^2}}
\int_{-1}^1 \frac{dt }{\sqrt{(1-t^2)^{4-n}}} \int_0^\infty Q(w)  W(w)dw
\\& =\frac{4\omega_{n-2}}{\omega_n}\frac {2^{n-1}}{(1+\rho)^{n-1}}
\frac 1{\sqrt{1+\gamma^2}}\int_{-1}^1 \frac{dt }{\sqrt{(1-t^2)^{4-n}}}\int_0^{Z(\gamma t)} Q(w)  W(w)dw.
\end{split}\end{equation*}
Finally, applying one more obvious  integral  transform we obtain the following representation
\begin{equation*}\begin{split}
C(\rho\mathbf{e}_1;\ell_\tau) &=\frac{4\omega_{n-2}}{\omega_n}\frac{2^{n-1} }{(1+\rho)^{n-1}} \frac 1{\sqrt{1+\gamma^2}}
\int_{0}^1 \frac{\mathcal{P}_\rho(\gamma t) +\mathcal{P}_\rho(-\gamma t) }{\sqrt{(1-t^2)^{4-n}}} dt,
\end{split}\end{equation*}
where
\begin{equation*}
\mathcal{P}_\rho  (z) = \int_0^{Z(z)} Q(w) W(w)dw.
\end{equation*}
This is which  we wanted to prove in i).

ii) We have  $C(x;\mathbf{n}_x)  =  C(|x|\mathbf{e}_1;\mathbf{e}_1)$ (because   $\mathbf{n}_{\rho \mathbf{e}_1} = \mathbf{e}_1$).  The
result of this part   is contained in i) for $\gamma = 0$. Since
\begin{equation*}
\frac{\omega_{n-2}}{\omega_{n}}=\frac 1\pi \frac {\Gamma\left(\frac n2\right)}   {{\Gamma\left(\frac {n-2}2\right)}},
\quad \int_0^1 \frac {dt }{ \sqrt{(1-t^2)^{4-n}} } = \frac {\sqrt{\pi}}2
\frac{ \Gamma\left( \frac {n-2}2\right) }{ \Gamma\left( \frac {n-1}2\right)},
\end{equation*}
we have
\begin{equation*}\begin{split}
C(\rho\mathbf{e}_1;\mathbf{e_1})
&=\frac{4\omega_{n-2}}{\omega_n}\frac{2^{n-1}}{(1+\rho)^{n-1}} 2 \mathcal{P}_\rho(0)  \int_{0}^1 \frac{dt }{\sqrt{(1-t^2)^{4-n}}}
\\&=\frac{4} {\sqrt{\pi}} \frac{2^{n-1}}{(1+\rho)^{n-1}} \frac{\Gamma\left(\frac n2\right)}
{ \Gamma\left( \frac {n-1}2\right) } \mathcal{P}_\rho(0),
\end{split}\end{equation*}
where
\begin{equation*}
\mathcal {P}_\rho  (0) = \int_0^{\frac {\sqrt{1-\alpha_\rho^2}}{1-\alpha_\rho}}
\frac { (n-\beta_\rho  - \beta_\rho w^2) w^{n-2} dw}  {(1+w^2)^{\frac n2+1}(1+\kappa_\rho^2 w^2)^{ \frac n2-1}}.
\end{equation*}
Since $\alpha_\rho=\rho(n-2)/n$, it is easy to obtain ${\sqrt{1-\alpha_\rho^2}}/({1-\alpha_\rho})=w_\rho$ and $(n-\beta_\rho)/\beta_\rho
= w_\rho^2$, which implies the representation of $C(x;\mathbf{n}_x)$ given in this lemma.

iii)  For  $\tau = \pi/2$  we have
\begin{equation*}\begin{split}
&C(\rho\mathbf{e}_1;\mathbf{t}_{\rho\mathbf{e}_1})
\\& = \frac{4\omega_{n-2}}{\omega_n}\frac1{1+\rho}
\int_0^\pi\sin^{n-3}\varphi d\varphi \int_0^{\frac \pi 2} |\mathcal {G}(\vartheta,\varphi,\gamma) | \mathcal{S}(\vartheta)
d\vartheta
\\& =\frac{4\omega_{n-2}}{\omega_n}\frac {2^{n-2}}{(1+\rho)^{n-1}}
\int_{-1}^1 \frac{|nt|dt }{\sqrt{(1-t^2)^{4-n}}} \int_0^\infty  \frac {w^{n-1}dw}
{ (1+w^2)^{ \frac n2+1}(1+\kappa_\rho^2w^2)^{\frac n2-1}}.
\end{split}\end{equation*}
Since
\begin{equation*}
\int_{-1}^1 \frac{|nt|dt }{\sqrt{(1-t^2)^{4-n}}} = 2n\int_{0}^1 \frac{ t dt }{\sqrt{(1-t^2)^{4-n}}}  = \frac {2n}{n-2},
\end{equation*}
by straightforward  calculations we find
\begin{equation*}
C(\rho\mathbf{e}_1;\mathbf{t}_{\rho\mathbf{e}_1}) =  \frac {2n}\pi \frac {2^{n-1}}{(1+\rho)^{n-1}}
  \int_0^\infty  \frac {w^{n-1}dw} { (1+w^2)^{ \frac n2+1}(1+\kappa_\rho^2w^2)^{\frac n2-1}}.
\end{equation*}
\end{proof}

\subsection{Gradient estimates   in $\mathbf{R}^n_+$}
In  Lemma \ref{LE.DOUBLE.2}  take  $\rho=1$.  Then we have $\beta_1 =1$  and
\begin{equation*}\begin{split}
\mathcal {G}(\varphi,\vartheta,\tau)&=(n\cos^2\vartheta - 1)\cos\tau + (n \sin \vartheta \cos\vartheta\cos \varphi)\sin\tau.
\\& =\frac {(n\cos ^2\vartheta - 1)    +     n\gamma \sin \vartheta \cos\vartheta\cos \varphi}{\sqrt{1+\gamma^2}},
\quad \gamma = \tan\tau.
\end{split}\end{equation*}
After  a short calculation one finds
\begin{equation*}
\mathcal{S}(\vartheta)  =  \sin^{n-2}\vartheta.
\end{equation*}
It   follows that
\begin{equation*}\begin{split}
{C}(\mathbf{e}_1; \ell_\tau)
=\frac{2\omega_{n-2}}{\omega_n}\frac 1{\sqrt{1+\gamma^2}}
\int_0^\pi\sin^{n-3}\varphi d\varphi
\int_0^{\pi/2} |\mathcal {G}(\varphi,\vartheta,\gamma)  |\sin^{n-2}\vartheta d\vartheta,
\end{split}\end{equation*}
where this time we use  the notation  $\mathcal{G}(\varphi,\vartheta,\gamma)$ for
\begin{equation*}
\mathcal {G}(\varphi,\vartheta,\gamma) = (n\cos ^2\vartheta - 1)    +     n\gamma \sin \vartheta \cos\vartheta\cos \varphi,
\end{equation*}
as in \cite{KM.DCDS}.  The numbers ${C}(\ell_\tau):= {C}(\mathbf{e}_1; \ell_\tau)$ are involved in the optimal  gradient estimates  of
bounded harmonic functions in the half--space $\mathbf {R}^n_+$ in  appropriate directions,       as   G. Kresin  and V. Maz'ya proved.

Regarding the above remark we will find $\mathcal{P}_1(z)$ from Theorem \ref{TH.FINAL}. It is  possible  to obtain an   integral--free
expression for $\mathcal{P}_1(z)$.  For $\rho=1$ we have $\alpha_1= (n-2)/n$,  $\beta_1=1$ and  $\kappa_1= 0$. Therefore,        since
\begin{equation*}
\int \frac{ (n - 1 + n z w - w^2)w^{n - 2}dw }{(1 + w^2)^{\frac n2 + 1}}= \frac { w^{n-1}(1 + zw)}{ (1 + w^2)^{\frac n2}},
\end{equation*}
we have
\begin{equation*}\begin{split}
\mathcal{P}_1(z) &= \int_0^{\frac {z+\sqrt{z^2 +1-\alpha_1^2}}{1-\alpha_1}}
\frac { (n-1  +n z w    -  w^2)w^{n-2} dw}{(1+w^2)^{\frac n2+1} } =  \frac {Z^{n - 1} (1 + z Z)}{(1 + Z^2)^{\frac n2}},
\end{split}\end{equation*}
where we have denoted
\begin{equation*}
Z    =    Z(z) = \frac {z+\sqrt{z^2 +1-\alpha_1^2}}{1-\alpha_1};\quad\text{then}\quad  z = z(Z) = \frac { 1-n +Z^2}{n Z}.
\end{equation*}
Since
\begin{equation*}
1+z Z   = 1+   \frac { 1-n +Z^2}{nZ}  Z= 1+  \frac { 1-n +Z^2}{n } =   \frac { 1  +Z^2}{n },
\end{equation*}
it follows
\begin{equation*}\begin{split}
\mathcal{P}_1(z)&=\frac {Z^{n-1}(1+ z Z )}{(1 + Z^2)^{\frac n2}}      = \frac 1n \frac{Z^{n-1}}{(1+Z^2)^{\frac n2-1}}.
\end{split}\end{equation*}
It is not hard to obtain
\begin{equation*}
\mathcal{P}_1(z)=\frac{(n-1)^{(n-1)}}n\mathcal {P} (y),
\end{equation*}
where
\begin{equation*}
y = \frac {n z}{2\sqrt{n-1}}\ \ \text{and}\ \ \mathcal{P}(y) 
= \frac {(y+\sqrt{y^2+1})^{n-1}}{(1+(n-1) (y+\sqrt{y^2+1})^2)^{\frac n2-1}}.
\end{equation*}

The following result is obtained in \cite{KM.DCDS} with $y$ and $\mathcal{P}(y)$ instead of $z$ and $\mathcal{P}_1(z)$ in the integral
expression,  but  the  formulation which follows  is  more  convenient  for us.

\begin{proposition}[Cf. \cite{KM.DCDS}]\label{PR.DCDS}
Let $x=(x',x_n)\in\mathbf{R}^n_+$, and let $\ell$ be a unit vector. Let $\tau$ be the angle between the straight lines   determined by
the vectors $\mathbf{n}_x = - \mathbf{e}_n$  and $\ell$.

The optimal  coefficient $\mathcal {C}(x;\ell)$ for the estimate
\begin{equation*}
\left|\frac {\partial U(x)}{\partial \ell}\right|\le \mathcal {C}(x;\ell)   \sup_{y\in\mathbf{R}^n_+}|U(y)|
\end{equation*}
may be represented as
\begin{equation*}
\mathcal{C}(x;\ell) = x_n^{-1}C(\ell),
\end{equation*}
where
\begin{equation*}\begin{split}
C(\ell)& = \frac{4\omega_{n-2}}{\omega_n} \frac 1{\sqrt{1+\gamma^2}}
\int_{0}^1 \frac{ \mathcal{P}_{1}(\gamma t) +\mathcal{P}_{1}(-\gamma t)  }{\sqrt{(1-t^2)^{4-n}}}dt,\quad  \gamma=\tan \tau,
\end{split}\end{equation*}
if $0\le  \tau<\pi/2$.

For the estimate  in the direction  $\mathbf{n}_x$ we have
\begin{equation*}
\left|\frac{\partial U(x)}{\partial \mathbf{n}_x}\right|\le
\frac 4{\sqrt{\pi}} \frac {(n-1)^{ \frac {n-1}2 } } { n^{\frac n2} }\frac {\Gamma\left(\frac n2\right)}
{\Gamma\left(\frac {n-1}2 \right)}\frac 1{x_n} \sup_{y\in\mathbf{R}^n_+} |U(y)|.
\end{equation*}

If $\mathbf{t}_x\bot \mathbf{n}_n$ is any tangential direction, i.e., a unit vector spanned by $\{\mathbf{e}_j: 1\le j\le n-1\}$, then
\begin{equation*}
\left|\frac{\partial U(x)}{\partial \mathbf{t}_x}\right|\le  \frac 2\pi\frac 1{x_n }\sup _{y\in\mathbf{R}^n_+}|U(y)|.
\end{equation*}

($U(y)$ is  among   bounded harmonic functions  in $\mathbf{R}^n_+$) 
\end{proposition}

\begin{remark}
Note that the estimate given in iii) resembles the inequality for harmonic functions  in the upper half--plane  which is stated in the
beginning of the paper. As in \cite{M.INDAG} it may be proved that the inequality is equivalent to the following one   $d_e(U(x),U(y))
\le 2/\pi  d_h(x,y)$    for  $x$,  $y\in \partial(\mathbf{R}^n_+  +(0,a))$,  where  $a>0$ is any number. Here $d_e$     stands for the
Euclidean distance, and  $d_h$  is   the hyperbolic  distance in $\mathbf{R}^n_+$ given by
\begin{equation*}
d_h (x,y) = \inf_\gamma \int_\gamma \rho(\omega) |d\omega|
\end{equation*}
(infimum   is taken over all rectifiable curves connecting $x$ and $y$), where $\rho(\omega)=\omega^{-1}_n$, $\omega\in\mathbf{R}^n_+$.
\end{remark}

The following corollary gives a connection between the coefficients in the two settings.

\begin{corollary}\label{CORO.BOUNDARY}
For every $\zeta\in\partial \mathbf{B}^n$ and  $\ell\in\partial \mathbf{B}^n$ we have
\begin{equation*}
\lim_{x\rightarrow \zeta} C(x;\ell) = C(\ell'),
\end{equation*}
where       $\ell'$       is a such  direction  that the angle between $\zeta$ and $\ell$ is the same as the angle between $\ell'$ and
$-\mathbf{e}_n$.
\end{corollary}

\section{Partial solution to the optimisation  problem}

\subsection{Remarks on the optimisation problem}
In this  section we consider the optimisation  problem
\begin{equation*}
\sup_{\gamma\ge0} C(\rho \mathbf{e}_1;\ell_\tau),
\end{equation*}
for $0<\rho\le 1$,         where $\gamma=\tan \tau$ (we set $\tan\pi/2=\infty$). According to the results of the previous section, the
Khavinson conjecture for the unit ball is equivalent to the statement that this problem has a solution at   $\gamma=0$ for every $\rho
\in(0,1)$. Our aim here is to prove this  statement for $\rho$ enough close to $1$.  The approach given here is based on the work of G.
Kresin and V. Maz'ya \cite{KM.DCDS}  where they proved that the optimization problem has  a solution at $\gamma=0$ for $\rho=1$, which
result is equivalent to the Khavinson type problem in the half--space setting.

Regarding Theorem \ref{TH.FINAL} and Proposition \ref{PR.DCDS} the optimisation problem we consider is equivalent to the following one
\begin{equation}\label{VAR.PSI}
\sup_{\gamma\ge 0}\frac 1{\sqrt{1+\gamma^2}}
\int_{0}^1 \frac{ \mathcal{P}_\rho(\gamma t) +\mathcal{P}_\rho(-\gamma t) }{\sqrt{(1-t^2)^{4-n}}} dt,
\end{equation}
where the main role is played by the function $\mathcal{P}_\rho (z)\, (0< \rho\le 1)$. Note the following fact.     If we are able  to
establish an inequality of the type
\begin{equation*}
\mathcal{P}_\rho (z)+ \mathcal{P}_\rho(-z)\le A(z)\  \  \text{for} \ \ z\in\mathbf{R},
\end{equation*}
together  with  the equality
\begin{equation*}
2\mathcal{P}_\rho(0) = A(0),
\end{equation*}
where $A(z)$ is a non--negative symmetric function, i.e., $A(z)=A(-z)$, and if the new optimisation problem
\begin{equation}\label{VAR.A}
\sup_{\gamma\ge 0}   \frac 1{\sqrt{1+\gamma^2}} \int_{0}^1 \frac{ A(\gamma t) }{\sqrt{(1-t^2)^{4-n}}} dt
\end{equation}
has a unique solution at $\gamma=0$,  then the same is true for the problem \eqref{VAR.PSI}.       Although not explicitly stated, the
preceding remark is crucial in resolving the optimisation problem for  $\rho=1$    in \cite{KM.DCDS}, along with an inequality   which
will be stated below in Proposition \ref{PR.INEQ.DCDS}.  The lemma which follows is inspired by the approach from the  mentioned paper.

\subsection{An auxiliary optimisation problem}
We will solve the problem \eqref{VAR.A} for the function
\begin{equation*}
A(z) = \sqrt{a z^2+b},\quad z\in\mathbf{R}.
\end{equation*}
This  result will  be very useful  in the sequel.

\begin{lemma}\label{LE.EXTREMAL}
Assume that $a$ and $b$ are positive numbers. If $a/b\le n-1$, then  the optimisation  problem
\begin{equation}\label{VAR.GAMMA}
\sup_{\gamma\ge 0}\frac 1 {\sqrt{1+\gamma^2}}
\int_0^1 \frac { \sqrt{a(\gamma  t)^2+b} } { \sqrt{(1-t^2)^{4-n}}}  dt
\end{equation}
has a solution for $\gamma =0$.   Moreover,   if $a/b<n-1$,  then $\gamma = 0$ is the unique solution to the problem \eqref{VAR.GAMMA}.
\end{lemma}

\begin{proof}
Straightforward  calculations give the following two relations
\begin{equation*}
\int_0^1 \frac {dt}{ \sqrt{(1-t^2)^{4-n}} }  = \frac {\sqrt{\pi}}2 \frac{ \Gamma\left( \frac {n-2}2\right) }
{ \Gamma\left( \frac {n-1}2\right) }
\end{equation*}
and
\begin{equation*}\begin{split}
\int_0^1\frac {a(\gamma  t)^2+b } { \sqrt{(1-t^2)^{4-n}}}dt
= \frac{ \sqrt{\pi}}2 \frac{ \Gamma\left(\frac {n-2}2\right)}
{  \Gamma\left(\frac {n-1}2\right)}\frac{ a \gamma^2 + b ( n-1) }{n-1}.
\end{split}\end{equation*}
Applying  the Cauchy--Schwarz inequality  we obtain
\begin{equation*}\begin{split}
\int_0^1  \frac {\sqrt{a(\gamma  t)^2+b}} {\sqrt{1+\gamma^2}\sqrt{(1-t^2)^{4-n}}}dt
& =\frac 1{\sqrt{1+\gamma^2}} \int_0^1 \frac 1{ \sqrt{(1-t^2)^{2-\frac n2}} }
\frac {\sqrt{a(\gamma  t)^2+b} } { \sqrt{(1-t^2)^{2-\frac n2}}}  dt
\\&\le \frac 1{\sqrt{1+\gamma^2}}\sqrt{\int_0^1 \frac {dt}{ \sqrt{(1-t^2)^{4-n}} } }
\sqrt{\int_0^1\frac {a(\gamma  t)^2+b } { \sqrt{(1-t^2)^{4-n}}}dt }
\\&= \frac 1{\sqrt{1+\gamma^2}}
\frac{\sqrt{\pi}}2 \frac{\Gamma\left(\frac {n-2}2\right) }{\Gamma\left( \frac {n-1}2\right) }
\sqrt{ \frac { a \gamma^2 + b ( n-1)} {n-1}}
\\&=   \frac{\sqrt{\pi}}2 \frac{\Gamma\left(\frac {n-2}2\right) }{\Gamma\left( \frac {n-1}2\right) }
\sqrt{ \frac { a \gamma^2 + b( n-1)} {(n-1)(1+\gamma^2)}}
\\&\le   \frac{\sqrt{\pi}}2 \frac{\Gamma\left(\frac {n-2}2\right) }{\Gamma\left( \frac {n-1}2\right) }\sqrt{b}
\end{split}\end{equation*}
for every  $\gamma\ge 0$. The second inequality above  holds  since the  function
\begin{equation*}
g(\gamma) =  \sqrt{  \frac {a \gamma^2 + b ( n-1) }{(n-1)(1+\gamma^2)} }
\end{equation*}
is decreasing in $\gamma\ge  0$,  if $a/b\le n-1$.  This follows  since
\begin{equation*}
\frac {d}{d\gamma}  g(\gamma) =
\frac { (a -  (n-1)b)\gamma } {(n-1) (1 + \gamma^2)^2g(\gamma)}\le 0,\quad \gamma\ge  0.
\end{equation*}
On   the other hand,  for  $\gamma=0$ we have the equality sign everywhere in the above sequence of estimates.

If    $a/b<n-1$, then the extremum is achieved only for $\gamma=0$, since in this case the function $g(\gamma)$ is strictly decreasing
in $\gamma \ge0$.
\end{proof}

\subsection{An inequality of G. Kresin and V. Maz'ya}
The following nontrivial inequality is established  in \cite{KM.DCDS}.

\begin{proposition}[Cf. \cite{KM.DCDS}]\label{PR.INEQ.DCDS}
Let
\begin{equation*}
\mathcal{P}(y) = \frac {(y+\sqrt{y^2+1})^{n-1}}{(1+(n-1) (y+\sqrt{y^2+1})^2)^{\frac n2-1}}
\end{equation*}
for  $y\in\mathbf{R}$.  Then
\begin{equation*}
\mathcal{P}(y)^2 + \mathcal{P}(-y)^2\le \frac {4(n-1)(3n-2)y^2  +   2n^2}{n^n}.
\end{equation*}
\end{proposition}

As a consequence of the above proposition we will  derive the following inequality suitable for our needs.

\begin{lemma}\label{LE.INEQ.P1}
If  $K = (3n-2)/4$, then
\begin{equation*}%\label{INEQ.RHO=1}
\mathcal{P}_1(z) + \mathcal{P}_1(-z)\le 2\mathcal{P}_1(0) \sqrt{K z^2 +1}.
\end{equation*}
The inequality is strict, unless for  $z =0$.
\end{lemma}

Note  that $K<n-1$. Actually, the inequality given in the lemma was established  in \cite{KM.DCDS} in a different  form     (via   the
function  $\mathcal {P}(y)$).   See the  inequality (5.17) on the page 438 there. For the sake of completeness we will write a   proof
of this fact.

\begin{proof}[Proof of Lemma \ref{LE.INEQ.P1}]
Recall  first  that  we have
\begin{equation*}
\mathcal{P}_1(z)=\frac{(n-1)^{(n-1)}}n\mathcal {P} (y),\quad  y = \frac {n z}{2\sqrt{n-1}}.
\end{equation*}
Note that $\mathcal{P}_1( -z) = \mathcal{P}(-y)$. Therefore, the inequality in this lemma may be rewritten as
\begin{equation}\label{INEQ.P}
\mathcal{P}(y)+\mathcal{P}(-y) \le 2\mathcal{P}(0)\sqrt{K\frac{4(n-1)}{n^2}y^2 +1}.
\end{equation}
By Proposition \ref{PR.INEQ.DCDS}  we have
\begin{equation*}
\mathcal{P}(y)^2+\mathcal{P}(-y)^2\le \frac{4(n-1)(3n-2)y^2 + 2n^2}{n^n}.
\end{equation*}
On the other hand, obviously,  there holds
\begin{equation*}
\mathcal{P}(y) \mathcal{P}(-y) =\frac 1 {( 4(n-1)y^2+n^2 )^{n/2-1}}\le \frac 1{n^{n-2}}.
\end{equation*}
This inequality is strict except for $y  = 0$. Since $\mathcal{P}(0)^2 = {n^{ 2-n}}$, it follows that
\begin{equation*}\begin{split}
(\mathcal{P}(y) +\mathcal{P}(-y))^2& \le \frac{4(n-1)(3n-2)y^2 +4n^2}{n^n} \\&= 4\mathcal{P}(0)^2\left(\frac{3n-2}{4}
\frac{4(n-1)}{n^2} y^2+1\right),
\end{split}\end{equation*}
which  is inequality  \eqref{INEQ.P} for   $K=({3n-2})/4$.
\end{proof}

\subsection{The optimisation problem for $\rho\ne 1$}
Let   us first briefly  discus the optimisation problem in the case $\rho=1$.     It may be solved via the inequality given  in  Lemma
\ref{LE.INEQ.P1}   and our Lemma \ref{LE.EXTREMAL}. Therefore,  the  problem  for  $\rho =1$ has  a unique   solution at  $\gamma = 0$.
Although Lemma \ref{LE.EXTREMAL} is not proved in \cite{KM.DCDS},   the authors solved the  problem on the base  of the  same approach.

As we have observed the inequality
\begin{equation}\label{INEQ.RHO}
\frac{\mathcal{P}_\rho(z)  +  \mathcal{P}_\rho(-z)}{2\mathcal{P}_\rho(0)}\le\sqrt{ K z^2 +1},\quad z\in\mathbf{R},
\end{equation}
is valid for  $\rho=1$, if we take $K=({3n-2})/4< n-1$.  The equality sign (when  $\rho=1$)  attains  only at  $z=0$.   Therefore, for
every  $z>0$ we  have the strict inequality  above.

Let   $M>0$ be an arbitrary big number. Our aim is to show that the inequality \eqref{INEQ.RHO}  is  valid for every  $z\in [0,M]$, if
$\rho$ is sufficiently near $1$. In that approach we first use the following simple

\begin{lemma}\label{LE.2DER}
Let  $F(z)$ and   $G(z)$ be two $C^2$-smooth functions defined in a neighborhood of a segment $[0,l]$. If $F(0)=G(0)$, $F'(0)=G'(0)=0$,
and $F''(z)\le G''(z)$ for all $z\in [0,l]$, then also $F(z)\le G(z)$ for all $z\in [0,l]$.

Particularly, if $F''(0)<G''(0)$, then we have $F''(z)\le G''(z)$ for all $z\in[0,\varepsilon]$, where $\varepsilon>0$ is sufficiently
small.
\end{lemma}

\begin{proof}
Let us consider the function $H(z)=F(z)-G(z)$ for $z\in [0,l]$. We have $H(0)=H'(0) = 0$ and
\begin{equation*}
H''(z)\le 0,\quad z\in [0,l].
\end{equation*}
Therefore,  for every $z\in[0,l]$ we obtain
\begin{equation*}
H'(z)=H'(z)-H'(0) \, = \int_0^z H''(w)dw\le0.
\end{equation*}
Similarly, we have
\begin{equation*}
H(z)  = H(z)-H(0)\, = \int_0^z H'(w)dw\le 0,
\end{equation*}
which proves the lemma.
\end{proof}

We   will prove now that the validity of the inequality \eqref{INEQ.RHO} for $\rho=1$ implies the inequality \eqref{INEQ.RHO} for   $z
\in[0,M]$, if $\rho$ is near $1$. Denote the left side in \eqref{INEQ.RHO} by $F(z)$ and the right side by $G(z)$.          Because of
symmetry we have $F'(0)=G'(0)=0$. Since for $\rho=1$ the inequality \eqref{INEQ.RHO} holds, we must have $F''(0)\le G''(0)$. Otherwise,
if the reverse inequality $F''(0)>G''(0)$ would be true, then we also have, in view of Lemma \ref{LE.2DER},  the reverse inequality in
\eqref{INEQ.RHO} for $\rho=1$ and for some values of $z$ close   to $0$,  which is incorrect. Moreover, we can achieve      the strict
inequality $F''(0) < G''(0)$ for $\rho=1$.  Indeed,    if the equality $F''(0)= G''(0)$ takes place, then we can slightly increase $K<
n-1$  so that  $F''(0) < G''(0)$.   This is possible because
\begin{equation*}
G''(z) = \frac{K}{\sqrt {(1+Kz^2)^{3}}},
\end{equation*}
and  therefore $G''(0) = K$.  We assume in the sequel  that this is done.  Because of continuity the same  inequality for   the second
derivatives remains valid if $\rho$ is near $1$. Applying again Lemma \ref{LE.2DER}, we conclude that the inequality   \eqref{INEQ.RHO}
holds for $z\in [0,\varepsilon]$, where $\varepsilon$ is sufficiently close to $0$. Since we have strict inequality in \eqref{INEQ.RHO}
for $z\in [\varepsilon, M]$, if $\rho=1$, it follows that this inequality   is true if $z$ belongs to the   same segment and if $\rho$
is close to $1$. Altogether, we have proved \eqref{INEQ.RHO} for $z\in [0,M]$ and $\rho $ close to $1$.

For the sake of  simplicity introduce  the function  $c(\rho,\gamma)=C(\rho\mathbf{e}_1;\ell_\tau)$, where $\gamma = \tan \tau$.   Let
us consider now our optimisation  problem.  First of all, we have $c(1,0)>c(1,\infty)$.       Because of continuity, this implies that
$c(1,0)>c(1,\gamma)$, if $\gamma\ge M$, where $M$ is big enough. The last inequality implies that $c(\rho,0)>c(\rho,\gamma)$       for
$\gamma\ge M$, if $\rho$ is close to $1$.  For $z\in[0,M]$ we have validity of \eqref{INEQ.RHO} (if $\rho$ is perhaps   closer to $1$).
In view of Lemma \ref{LE.EXTREMAL}, this implies that $c(\rho,0)> c(\rho,z)$,  $z\in [0,M]$. Therefore, the optimisation problem has a
unique solution  at  $\gamma=0$,    if $\rho$ is sufficiently close to $1$.

\begin{remark}
It seems that the inequality  given  in Lemma  \ref{LE.INEQ.P1}, which is,    in our approach, crucial  in  resolving the optimisation
problem, is not valid for every $\rho\in(0,1)$ (even for $\rho\in (0,0.98)$, if $n=3$).        This  suggests that the solution to the
optimisation  problem is much more harder in  general.
\end{remark}

\subsection{A boundary result}
At the end of this section let us discus one boundary result.    We have proved that $\mathcal {C}(x)=\mathcal{C}(x;\mathbf{n}_x)$, if
$x\in\mathbf{B}^n$ is near the boundary $\partial \mathbf{B}^n$.                              It follows that $(1-|x|)\mathcal {C}(x)=
(1-|x|)\mathcal{C}(x;\mathbf{n}_x)$, which  may be rewritten as
\begin{equation*}
\sup_{U,\, \|U\|_\infty\le 1}  (1- |x|)\left|\nabla U(x)\right|
\ \  =   \sup_{U,\, \|U\|_\infty\le 1} (1- |x|) \left|\left<\nabla U(x),\mathbf{n}_x\right>\right|, \quad \text{if}\quad |x|\approx 1.
\end{equation*}
Letting $x\rightarrow  \zeta\in\partial\mathbf{B}^n$ above,  and bearing in mind that the both sides depend only on $|x|$,   we derive
\begin{equation*}\begin{split}
\lim_{\mathbf{B}^n\ni x\rightarrow \zeta}\sup_{U,\, \|U\|_\infty\le 1} (1- |x|)\left|\nabla U(x)\right|\ \, &=
\lim_{\mathbf{B}^n\ni x\rightarrow \zeta}\sup_{U,\, \|U\|_\infty\le 1} (1- |x|) \left|\left<\nabla U(x),
\mathbf{n}_x\right>\right|
\\&=  \lim_{\mathbf{B}^n\ni x\rightarrow \zeta} C(x;\mathbf{n}_x)=C(-\mathbf{e}_n)
\\& = \frac 4 {\sqrt{\pi}} \frac {(n-1)^{ \frac {n-1}2} } { n^{\frac n2} }\frac {\Gamma\left(\frac n2\right)}
{\Gamma\left(\frac {n-1}2\right)}
\end{split}\end{equation*}
(see Corollary \ref{CORO.BOUNDARY}).         This  boundary relation is a special case of the corresponding result  in  \cite{KM.DCDS}
obtained for more   general  domains in $\mathbf{R}^n$.

\end{document}